\documentclass{fic-l}

\newtheorem{theorem}{Theorem}[section]
\newtheorem{lemma}[theorem]{Lemma}
\newtheorem{proposition}[theorem]{Proposition}
\newtheorem{corollary}[theorem]{Corollary}
\newtheorem{question}[theorem]{Question}

\theoremstyle{definition}
\newtheorem{definition}[theorem]{Definition}
\newtheorem{example}[theorem]{Example}

\theoremstyle{remark}
\newtheorem{remark}[theorem]{Remark}

\numberwithin{equation}{section}

\begin{document}

\title{Existence of symplectic surfaces}

\author{Tian-Jun Li}
\address{School  of Mathematics\\  University of Minnesota\\ Minneapolis, MN 55455}
\email{tjli@math.umn.edu}
\thanks{The author is supported in part by NSF grant 0435099 and the McKnight fellowship.}

\date{September 20, 2004}


\begin{abstract}
In this paper we show that every degree 2 homology class of a 2n-dimensional symplectic manifold 
is represented by an immersed symplectic surface if it has positive symplectic area. Moreover,
the symplectic surface can be chosen to be embedded if  2n is at least 6. We also analyze the
additional conditions under which embedded symplectic representatives exist in dimension 4.
\end{abstract}

\maketitle

\section{Introduction}

In this paper we will address the question which 2 dimensional homology class of a symplectic manifold $(M,\omega)$
can be  represented by an embedded  symplectic submanifold.
To study this problem, we first explore Gromov's beautiful h-principles on symplectic embeddings and immersions. 
   We are able to  give a complete answer in dimensions 6 and above 
 and derive some interesting consequences. 
 Clearly, a necessary condition for a class
$A\in H_2(M;{\Bbb Z})$ to be represented by an embedded (or an immersed) symplectic 
surface is that $\omega(A)>0$. We call such a class a $\omega-$positive class.
Our observation is  that every $\omega-$positive class is in fact  effective.

\begin  {theorem} Suppose $(M,\omega)$ is a symplectic manifold
of dimension $2n$. Let $A$ be any $\omega-$positive class in $H_2(M;{\Bbb
Z})$. Then  
\begin{enumerate}

\item $A$ is represented by a connected embedded
$\omega-$symplectic surface if $2n\geq 6$. 

\item $A$ is represented by a connected immersed $\omega-$symplectic surface if $2n\geq 4$.

\end{enumerate}
\end{theorem}
Here, an embedding $\Sigma\subset M$ is called symplectic if $\omega$ restricts
to a symplectic form on $\Sigma$. Similarly, an immersion $i:\sigma\longrightarrow M$
is called symplectic if $i^*\omega$ is a symplectic form on $\Sigma$. This result should
also be known to some experts, e.g. see the recent preprint \cite{Le}.

We also analyze in dimension 4 when the immersed symplectic surface can be made embedded.
 In view of the normal connected sum operation in
~\cite{Go} and ~\cite{MW}, it is a very important question.  
 We will analyze some obstructions and also describe several
 constructions. 

\medskip

The author wishes to thank L. Ein, C. Leung, B. H. Li and Y. Ruan for very helpful discussions. The author also
wishes to thank the referee for the careful readings and pointing out the reference \cite{Le}. 

\section {Symplectic surfaces via h-principles}

\subsection {Smoothly embedded surfaces, monomorphisms and h-principles}

In this section $\Sigma$ is a closed oriented surface.
We first collect some basic and well-known facts of representing a homology class of degree
2 by $\Sigma$.

\begin {lemma} 
(see Hopf ~\cite{Ho}) Let $M$ be a connected oriented smooth
manifold of dimension $2n$ and $A$ a homology class of degree 2.
Then 
 $A$ is represented by a continuous map  from an oriented surface $\Sigma$
to $M$.  
\end{lemma}

\begin{lemma} 
Let $M, A$ be as in Lemma 2.1 and suppose
that $A$ is
 represented by a continuous map $f$  from an oriented surface $\Sigma$
to $M$. Then
\begin{enumerate}
  
\item  When $2n\geq 6$,  $f:\Sigma\longrightarrow M$ 
is homotopic to  an embedding of $\Sigma$ into $M$.

\item 
When $2n=4$,  $f:\Sigma\longrightarrow M$ 
 is homotopic to  an immersion $g:\Sigma\longrightarrow M$; and the   
double points of the immersed surface $g(\Sigma)$ 
 can be eliminated to get an embedded surface 
of higher genus still representing $A$. 
\end{enumerate}
\end{lemma}

Lemma 2.2 follows from the standard transversality theory: the
continuous map $f$ can be first perturbed to an immersion when $2n\geq 2\cdot 2$;
furthermore the immersion can be approximated by
an embedding when $2n\geq 2\cdot 2+1$.

When $2n\geq 4$, any disconnected embedded 
(or immersed) surface  
 can be tubed to get a connected embedded (or immersed) surface
representing the same class. Thus, following from  Lemmas 2.1
and 2.2,  we have

\begin{proposition}
 (see Thom ~\cite{Th}) Let $M$ be a connected oriented smooth
manifold of dimension $2n\geq 4$ and $A$ a homology class of degree 2.
Then 
\begin{enumerate}

\item  $A$ is represented by a connected embedded surface. 

\item If $A$ is a spherical class, then $A$ is represented by
an embedded sphere if $2n\geq 6$ and an immersed sphere if $2n=4$. 
 
\end{enumerate}
\end{proposition}

Here, a class is said to be spherical if it is represented by    
a continuous map from $S^2$ to $M$. In other words, it is in the image
of the Hurewicz homomorphism $\pi_2(M)\longrightarrow H_2(M;{\Bbb
Z})$.

Now we study the 
the space of real and complex monomorphisms between
bundles over $\Sigma$.

\begin{lemma}
 Let $\Sigma$ be a closed oriented surface and $E$ and $F$ 
be complex vector bundles over $\Sigma$ with  real
dimensions $2$ and  $2n\geq 4$ respectively. Then the space of 
complex monomorphisms Mono$_{\Bbb C}(E,F)$ is non-empty and connected. 
And  the space of 
real monomorphisms Mono$_{\Bbb R}(E,F)$ is non-empty and connected if $2n\geq 6$.

\end{lemma}

\begin{proof}  Let us first consider the bundle HOM$_{\Bbb C}(E, F)$ of complex homomorphisms. The stratified subbundle of singular 
homomorphisms has complex codimension $n$. By transversality, if
$2n$ is bigger than the real dimension of the base which is $2$,  then a generic complex homomorphism
is a complex monomorphism. And if $2n>2+1$, then every generic complex homomorphism over
$\Sigma\times [0,1]$ is a complex monomorphism. Thus the claim about Mono$_{\Bbb C}(E,F)$ is proved. 

Let us now consider the bundle HOM$_{\Bbb R}(E, F)$ of real homomorphisms. The stratified subbundle of singular 
homomorphisms now has real codimension $[2-(2-1)][2n-(2-1)]=2n-2+1=2n-1$. By  transversality, if $2<2n-1$, i.e. $2n\geq 4$
then a generic homomorphism is a monomorphism. Furthermore, if $2+1<2n-1$, then
a generic path of homomorphisms consists of monomorphisms only. 
The proof is now complete.
\end{proof}

\begin{remark}
When  $2n=4$,  a simple obstruction theory argument
can be used to show that 
  the  connected components of the real monomorphisms 
Mono$_{\Bbb R}(E,F)$ 
are indexed by the Euler number of the quotient bundle, which can be any integer
congruent to $c_1(F)-c_1(E)$ modulo 2.  For a complex monomorphism, the Euler number
of the quotient bundle is of course equal to  $c_1(F)-c_1(E)$. 
For more on complex and real monomorphisms see ~\cite{K}.
\end{remark}

 Recall  the notion of the normal Euler number of an immersion
of an oriented surface in an oriented 4-manifold. 
Given an immersion $f:\Sigma\longrightarrow M$, the differential $Df$
is a monomorphism from $T\Sigma$ to the pull-back bundle
$f^*TM$. 
 The normal bundle $N(f)$ of $f$ is then
the quotient bundle of $f^*TM$ by $Df(T\Sigma)$, and the Euler number of $N(f)$
is called the normal Euler number of $f$. 
Suppose $A$ is represented by a continuous map from $\Sigma$. 
We can apply Remark 2.5 and Hirsch's fundamental result on immersions in ~\cite{Hi}
(which can be considered as the h-principle for smooth immersions)
to realize $A$ by an immersion from $\Sigma$ to $M$ with any normal Euler number
as long as it is congruent to $A\cdot A$ modulo 2.

\medskip
 Now we state the relevant h-principles on symplectic embeddings and immersions
(see ~\cite{EM}). We first need to  introduce several definitions. Let $(M,\omega_M)$ and $(V,\omega_V)$
be symplectic
manifolds of dimensions $2n$ and $2q$ respectively. Let $f:V\longrightarrow M$
be a continuous map and $F:TV\longrightarrow TM$ be a real homomorphism which covers $f$.  The map $f$ is called  an isosymplectic immersion
if $f^*\omega_M=\omega_V$, and it is called a isosymplectic embedding if
it is also a smooth embedding. Notice that a symplectic immersion is
automatically a smooth immersion. 
The homomorphism $F$ is called symplectic if
$F^*\omega_M$ is non-degenerate and 
$f^*[\omega_M]=[\omega_V]$, and $F$ 
 is called isosymplectic if it is symplectic and
$F^*\omega_M=\omega_V$. Notice that symplectic and isosymplectic homomorphisms
are necessarilly monomorphisms. 

We begin with the embedded version.

\begin{theorem} [Gromov's isosymplectic embedding theorem]
 Let $(M,\omega_M)$ and $(V,\omega_V)$
be symplectic
manifolds of dimensions $2n$ and $2q$ respectively. Suppose that an
embedding
$f_0:V\longrightarrow M$ satisfies the cohomological condition
$f_0^*[\omega_M]=[\omega_V]$, and that the differential $F_0=Df_0$ is
homotopic
via a homotopy of monomorphisms
$F_t:TV\longrightarrow TM$ covering $f_0$
to an isosymplectic homomorphism $F_1:TV\longrightarrow TM$.  If $V$ is
closed
and $2n\geq  2q+4$, then there exists an isotopy $f_t:V\longrightarrow M$
such that the embedding $f_1:V\longrightarrow M$ is isosymplectic
and the differential $Df_1$ is homotopic to $F_1$ through
isosymplectic homomorphisms (with varying base maps). Moreover, one can choose the isotopy
$f_t$ to be arbitrarily $C^0-$close to $f_0$.  

\end{theorem}
Next is the  immersed version.

\begin{theorem} [Gromov's isosymplectic immersion theorem]
 Let $(M,\omega_M)$ and $(V,\omega_V)$
be symplectic
manifolds of dimensions $2n$ and $2q$ respectively. Suppose a continuous
map
$f_0:V\longrightarrow M$ satisfies the cohomological condition
$f_0^*[\omega_M]=[\omega_V]$, and that $f_0$ is covered by
an isosymplectic homomorphism $F:TV\longrightarrow TM$.
  If $V$ is closed
and $2n\geq  2q+2$, then there exists a homotopy $f_t:V\longrightarrow M$
such that  $f_1:V\longrightarrow M$ is an isosymplectic immersion,
i.e. $f_1^*\omega_M=\omega_V$, 
and the differential $Df_1$ is homotopic to $F$ through
isosymplectic homomorphisms. 

\end{theorem}

 Notice that, unlike the embedded case, $f_0$ is not assumed to be an immersion.
And even when $f_0$ is an immersion,  $F$ is not necessarilly  
homotopic to $Df_0$ via a homotopy of real  monomorphisms.

In general,  it is not easy to verify
the conditions required in these two theorems. However, for embeddings and immersions of 
symplectic surfaces, we
can effectively use almost complex structures to achieve these conditions. 

\subsection {Existence and some consequences}

We first briefly review how almost complex structures come into play.
Given a non-degenerate $2-$form $\tau$ and an almost complex structure $J$ on $M$, we say 
$\tau$ and $J$ are tamed by each other if 
$\tau(v, Jv)>0$ for any non-zero tangent vector $v$.  
Notice that  a closed $2-$form on $M$ is a symplectic form if and only if
it is tamed by $J\in {\mathcal J}_{\omega}$. 
Let $\omega$ be a symplectic structure on $M$, and
${\mathcal J}_{\omega}$ be the space of the almost complex structures on $M$
tamed by ${\omega}$. 
Then for any $J\in {\mathcal J}_{\omega}$, an embedded or immersed $J-$holomorphic submanifold is $\omega-$symplectic. 
Conversely, any embedded symplectic submanifold is $J-$holomorphic for some $J\in {\mathcal J}_{\omega}$.

\begin{proposition} Suppose $(M,\omega)$ is a symplectic manifold
of dimension $2n$ with symplectic canonical class $K_{\omega}$. Let $A$ be a homology class in 
$H_2(M;{\Bbb Z})$
with $\omega(A)>0$, $\Sigma$ be a connected oriented surface of genus
$g$,
and $f_0:\Sigma\longrightarrow M$ an embedding representing $A$.  Then 
\begin{enumerate}

\item There exists an embedding  $f_1:\Sigma\longrightarrow M$
such that $f_1(\Sigma)$ is a symplectic surface representing $A$
  if $2n\geq 6$.

\item There exists an immersion  $f_1:\Sigma\longrightarrow M$
such that $f_1(\Sigma)$ is a symplectic surface representing $A$
if $2n= 4$. Moreover the normal Euler number of the immersion $f_1$  is given by the adjunction formula
$2g-2-K_{\omega}(A)$. 
\end{enumerate}
\end{proposition}

\begin{proof}  We first assume  $2n\geq 6$ so that the codimension 4 condition in Theorem 2.6 is satisfied.

 Let $j$ be a complex structure on $\Sigma$ and $J$ be an almost complex structure on $M$ 
tamed by $\omega_M$. 
  By Lemma 2.4, we can find
a complex monomorphism $F_{1\over 2}$ which is homotopic to the real monomorphism
$F_0=Df_0$ via a homotopy of
real
monomorphisms covering $f_0$.

To find an isosymplectic homomorphisms  required in Theorem 2.6, we need the following two lemmas.
  
\begin{lemma}
For any complex monomorphism
$F:(T\Sigma,j)\longrightarrow (TM, J)$ covering
$f_0:\Sigma\longrightarrow M$,  $F^*\omega_M$ 
is non-degenerate
on $T\Sigma$ and is a closed 2-form with $\int_{\Sigma}F^*\omega_M>0$. 
\end{lemma}

\begin{proof} $F^*\omega_M$ is closed on $\Sigma$ since it is
a 2-form on a 2-manifold. 
Since $J$ is tamed by $\omega_M$,  we have 
$\omega_M(v, Jv)>0$ for any $x$ in $M$ and any non-zero $v\in T_xM$. 
Since $F$ is a complex
monomorphism, we have $F(v)\ne 0$
in $T_{f_0(z)}M$, and 
$$F^*\omega_M(v,jv)=\omega_M(F(v), F(jv))=\omega_M(F(v), J(F(v)))>0,$$
 for any non-zero $v\in T_z\Sigma$.
\end{proof}

 \begin{lemma} Let $G_0$ be a complex monomorphism from
 $(T\Sigma,j)$
to $(TM, J)$, and $R$ be any positive real number. Then $G_0$
 is homotopic via a homotopy of complex monomorphisms to a complex
 monomorphism
$G_1$ such that 
$$\int_{\Sigma}G_1^*\omega_M=R.$$
\end{lemma}

\begin{proof}
 By Lemma 2.9, $r=\int_{\Sigma}G_0^*\omega_M>0$.
Since, for any positive real number $t$,    $tG$ is still a complex monomorphism with $(tG)^*\omega=t(G^*\omega)$, we find that
 $$G_t=[(1-t)+t{R\over r}] G_0$$
 is a required homotopy.
\end{proof}

Applying  Lemma 2.10 to $G_0=F_{1\over 2}$ and $R=\omega_M(A)$, we obtain a complex
monomorphism $F_1=G_1$ which is homotopic to $F_{1\over 2}$ via a homotopy of
(complex)
monomorphisms covering $f_0$ and satisfies
$$\int_{\Sigma}F_1^*\omega_M=\omega_M(A).$$ 
 Notice that the monomorphism $F_{1}$ is  homotopic to $F_0=Df_0$ via a homotopy of
real
monomorphisms covering $f_0$ since $F_{1\over 2}$ is assumed to be so.
Now if we let $\omega_{\Sigma}=F_1^*\omega_M$, then $[\omega_{\Sigma}]=f_0^*[\omega_M]$,
and moreover $\omega_{\Sigma}$ is a
symplectic
form on $\Sigma$ by Lemma 2.9.

We have shown that $F_t$ is the required homotopy, therefore we can apply 
Theorem 2.6 to $(\Sigma, \omega_{\Sigma}), (M, \omega_M),
f_0$ and  $F_t$ to  conclude
that $f_0$ is isotopic to an embedding $f_1$ such that 
$f_1^*
\omega_M=\omega_{\Sigma}$. In particular, $f_1(\Sigma)$ is
an embedded $\omega_M-$symplectic surface representing the class $A$. 

\medskip

Now we assume that $2n=4$. 
Since the  isosymplectic immersion theorem only requires
codimension 2, 
it can be applied to this case.

By Lemma 2.4, there exists a complex monomorphim 
$G:(T\Sigma,j)\longrightarrow (TM,J)$ covering $f_0$ whose quotient complex line bundle
has Euler number
$$\int_{\Sigma} f_0^*c_1(TM,J)-\int_{\Sigma}c_1(T\Sigma,j)=-K_{\omega}(A)-(2-2g).$$  
Applying Lemma 2.10 to $G_0=G$ and $R=\omega_M(A)$,  we obtain a complex
monomorphism $F=G_1$ which satisfies
$$\int_{\Sigma}F^*\omega_M=\omega_M(A).$$ 
Now let $\omega_{\Sigma}=F^*\omega_M$, then $[\omega_{\Sigma}]=f_0^*[\omega_M]$,
and $\omega_{\Sigma}$ is a
symplectic
form on $\Sigma$ by Lemma 2.9.   

We have shown that $F$ is the required isosymplectic homomorphism, and therefore we can apply 
Theorem 2.7 to $(\Sigma, \omega_{\Sigma}), (M, \omega_M),
f_0$ and  $F$ to  conclude
that $f_0$ is homotopic to an immersion $f_1$ such that 
$f_1^*
\omega_M=\omega_{\Sigma}$. In particular, $f_1(\Sigma)$ is
an immersed $\omega_M-$symplectic surface representing the class $A$.
The proof of Proposition 2.8 is now complete. 
\end{proof}

\begin{proof} of Theorem 1. It  immediately follows from Lemma 2.1, Proposition 2.8 and the first part of Proposition 2.3. 
\end{proof}

\begin{remark}
The same method can be used to
show that the symplectic surfaces in Theorem 1.1 are symplectically isotopic. 
\end{remark}

As we will see in \S3.2, there are other constructions of symplectic submanifolds.
For an integral symplectic manifold $(M,\omega)$,   Donaldson ~\cite{Do}  uses an approximately holomorphic technique to construct
 codimension 2
sympletic submanifolds representing the 
Poincar\'e dual to high multiples of
the $[\omega]$.  In fact, Donaldson is able to 
show that any closed symplectic manifold contains symplectic submanifolds of
any even codimension. 
There have been various generalizations (see e.g. ~\cite{A},~\cite{MPS},~\cite{P}). However, the constructed 
symplectic submanifolds  only represent 
homology classes close to the Poincar\'e duals to the powers
of $[\omega]$. 
For 2-dimensional symplectic submanifolds, the h-principle
applies to  any class of degree
2. And  we have better control of the
genera of the symplectic surfaces. In particular, 
for manifolds which are not symplectically aspherical, we can find 
symplectic spheres. Recall that a symplectic manifold $(M,\omega)$ is said to be
symplectically
aspherical if $[\omega]$ vanishes on the image of the Hurewicz
homomorphism
$\pi_2(M)\longrightarrow H_2(M;{\Bbb Z})$. 
The following result is a direct consequence of  Proposition 2.8 and the second part of 
Proposition 2.3. 

\begin{corollary} Suppose $(M,\omega)$ is a symplectic manifold
of
dimension $2n$. Let $A$ be a spherical homology class in $H_2(M;{\Bbb
Z})$ with $\omega(A)>0$. Then $A$ is represented by an immersed
$\omega-$symplectic
$2-$sphere; and if $2n\geq 6$,   $A$ is represented by an embedded
$\omega-$symplectic
$2-$sphere. In particular there are symplectic spheres in $(M,\omega)$
if and only if it is not symplectically aspherical.  
\end{corollary}

Specifically,  when $M$ is simply
connected, by the Hurewicz Theorem,
 every class of degree 2 is 
spherical, and hence represented by 
 an embedded symplectic sphere if $2n\geq 6$ and an immersed symplectic sphere if $2n=4$. 

Since an embedded symplectic submanifold is $J-$holomorphic for some 
$J\in{\mathcal J}_{\omega}$, we also have the  existence of
embedded
pseudo-holomorphic curves. 

\begin{corollary}
Let   $(M, \omega)$ be a symplectic manifold of
  dimension at least 6, and $A$ a  2 dimensional 
homology class  with $\omega(A)>0$. Then $A$  is 
represented by an embedded $J-$holomorphic curve for some $J\in{\mathcal J}_{\omega}$.
 And if $M$ is simply connected, 
$A$ is represented 
by an embedded 
$J-$holomorphic sphere for some 
$J\in{\mathcal J}_{\omega}$.
\end{corollary}

 However, 
in dimension 4, the immersed symplectic surfaces obtained
in Proposition 2.8 cannot always
be chosen to be pseudo-holomorphic.
In fact, an immersed symplectic surface is  $J-$holomorphic for  some $J$
if and only if it is positively immersed, i.e. all the double points are positive
double points. On the other hand,  we will see from Lemma 3.16 that, for
a fixed class $A$, 
the existence of a simple pseudo-holomorphic curve (not necessarily embedded)
is equivalent to the existence of an embedded symplectic surface.

We end section 2.2 by showing that Theorem 1.1 implies a duality
between surface cones and symplectic cones over ${\Bbb Q}$. We first need to  introduce several
definitions. Let $V$ be a vector space over ${\Bbb Q}$ or ${\Bbb R}$
and $V^*$ be the dual space of $V$. 
If $U$ is a subset of $V$, 
the dual of $U$ is the subset $U^*$ in $V^*$ given by
$$U^*=\{\alpha\in V^*|\alpha(v)>0 \hbox{ for any $v\in U$ }\}.$$
In the following, $V$ would either be $H^2(M;{\Bbb R})$ or $H^2(M;{\Bbb Q})$. 

Define the $\omega-$effective set ${\mathcal A}_{\omega}$ by
$${\mathcal A}_{\omega}=\{A\in H_2(M;{\Bbb Z})|A \hbox{ is represented by an
embedded $\omega-$symplectic surface}\}.$$
In terms of this definition, Theorem 1.1 is then simply the following
statement. 

\begin{corollary} If $2n\geq 6$, then we have
 ${\mathcal A}_{\omega}=[\omega]^* \cap H_2(M;{\Bbb Z})$, where
$[\omega]^*$ is the real dual of $[\omega]$. 
\end{corollary}

If we define the rational  $\omega-$surface cone 
${\mathcal S}^{\Bbb Q}_{\omega}$ (real $\omega-$surface cone 
${\mathcal S}^{\Bbb R}_{\omega}$) as the convex cone in
$H_2(M;{\Bbb Q})$ ($H_2(M;{\Bbb R})$) generated by ${\mathcal A}_{\omega}$,
then we have the following weaker version of Theorem 1.1.

\begin{corollary}  The $\omega-$surface cones are given by 
$${\mathcal S}^{\Bbb Q}_{\omega}=[\omega]^*\cap H_2(M;{\Bbb Q}),\quad \quad
{\mathcal S}^{\Bbb R}_{\omega}=[\omega]^*.
$$ 
\end{corollary}

\begin{proof} It is clear from Theorem 1.1 that   a rational point in $[\omega]^*$ is in ${\mathcal S}^{\Bbb Q}_{\omega}$. Since any class in ${\mathcal A}_{\omega}$
has positive $\omega-$area we always have ${\mathcal S}^{\Bbb R}_{\omega}\subset [\omega]^*$. Thus we obtain the  first equality. To further show that  
${\mathcal S}^{\Bbb R}_{\omega}=[\omega]^*$, we notice that $[\omega]^*$ is an
open subset of $H_2(M;{\Bbb R})$. Hence, as observed in
~\cite{B},   every class $v\in [\omega]^*$ can be written as
$v=\sum_i v_i$ with each $v_i$  on a rational ray in $[\omega]^*$. 
Thus $v\in {\mathcal S}^{\Bbb R}_{\omega}$ as each $v_i$ is in ${\mathcal S}
^{\Bbb Q}_{\omega}$.
\end{proof}

We can generalize the duality to a family of symplectic forms as follows.
 Let now $M$ be an oriented smooth manifold and
$\Omega(M)$ be the space of orientation-compatible 
symplectic structures on $M$. There is a natural ${\Bbb R}^+$ action
on $\Omega(M)$ obtained from multiplying a symplectic form by a positive real number. 
Let ${\mathcal W}$ be any subset of $\Omega(M)$
invariant under the ${\Bbb R}^+$ action. Then the  ${\mathcal W}-$symplectic cone
is defined as 
$${\mathcal C}_{\mathcal W}=\{[\omega]\in H^2(M;{\Bbb R})|      \omega\in {\mathcal
W}\},$$
 and the intersection of ${\mathcal C}_{\mathcal W}$ with
$H^2(M;{\Bbb Q})$ is called the rational ${\mathcal W}-$symplectic cone. 
The subset of $H_2(M;{\Bbb Z})$,
${\mathcal A}_{\mathcal W}=\cap_{\omega\in {\mathcal W}} {\mathcal A}_{\omega}$
is called the ${\mathcal W}-$effective set. 
We similarly define the rational ${\mathcal W}-$surface cone
${\mathcal S}_{\mathcal W}^{\Bbb Q}$ as the convex cone in
$H_2(M;{\Bbb Q})$ generated by ${\mathcal A}_{\mathcal W}$.  
We can similarly define ${\mathcal A}^{\hbox{imm}}_{\omega}$
and immersed versions of other concepts. 

Now we can state the duality between the rational ${\mathcal W}-$surface cone and the rational ${\mathcal W}-$symplectic cone. 

\begin{corollary} Suppose $M$ is a closed, oriented manifold
of dimension $2n$ admitting symplectic structures, i.e. $\Omega(M)$ is
nonempty. If $2n\geq 6$, then for
any subset ${\mathcal W}\subset \Omega(M)$, we have 
$${\mathcal A}_{\mathcal W}=\{A\in H_2(M;{\Bbb Z})|\omega(A)>0 \hbox{ for any }\omega\in {\mathcal W}\}.$$
Consequently, the rational ${\mathcal W}-$surface cone 
is dual to 
the rational ${\mathcal W}-$symplectic cone over ${\Bbb Q}$. 
\end{corollary}

\begin{remark}
 Notice that we do not in general  have the duality over ${\Bbb R}$. This is because when the quotient of ${\mathcal C}_{\mathcal W}$ under 
${\Bbb R}^+$ is non-compact,
the real dual of ${\mathcal C}_{\mathcal W}$ may have 
boundaries containing irrational rays. 
\end{remark}

When $2n=4$, 
the conclusion in Corollary 2.16  is still true for the {\it immersed} rational
$W-$surface cone. 
However, as can be seen from the next section,
 for the embedded surface cone,  we can only expect the duality 
to hold  for  minimal 4-manifolds with $b^+=1$.
Such a duality has been verified  in ~\cite{LLiu2}  for several classes
of such manifolds    
in the case ${\mathcal W}$ is $\Omega_K(M)=\{\omega\in \Omega(M)|K_{\omega}=K\}$ for some 
 $K\in H^2(M;{\Bbb Z})$.

\section{Embedded symplectic surfaces in 4-manifolds}

Let $(M,\omega)$ be a symplectic manifold of dimension 4.
One  distinctive feature of embedded symplectic surfaces in this dimension is the adjunction formula:  
Given a $\omega-$positive class $A$, if it is represented by
a connected symplectic surface, then the genus $g$ of such a surface  is uniquely determined 
by the adjunction formula, 
$$2g-2=K_{\omega}(A)+A\cdot A.$$
We call it the $\omega-$symplectic genus of $A$ and denote it by 
$g_{\omega}(A)$. 

In fact, the generalized Thom conjecture  (\cite{KM}, ~\cite{MST},~\cite{LLiu1}, and ~\cite{OS}) asserts that 
  $g_{\omega}(A)$ is smaller than or equal to the genus 
of any smoothly embedded  connected surface representing $A$. 
We should point out that the  adjunction formula also applies to
 a surface $\Sigma$ with several components $\Sigma_1,...,\Sigma_l$, provided that we define the genus
to be $\sum_{i=1}^l g(\Sigma_i)-(l-1)$.

We know from Theorem 1.1 that every $\omega-$positive class is represented by
a connected immersed symplectic surface. We would like to know which
$\omega-$positive class is represented by an embedded symplectic surface
(not necessarily connected). We will first analyze some obstructions
and then describe several constructions. 

In this section we will sometimes
identify a cohomology class in $H^2(M;{\Bbb Z})$ with its
Poincar\'e dual in $H_2(M;{\Bbb Z})$. For instance when we say
a cohomology class $e$ is realized by a surface $\Sigma$, it means that the Poinca\'e dual to
$e$ is represented by $\Sigma$.

We remark that,  for the isotopy problem of embedded symplectic surfaces in this dimension, there are both
uniqueness and non-uniqueness results (see ~\cite{FS} and ~\cite{ST}).

\subsection{Obstructions}

There
are several obstructions to represent a $\omega-$positive class by a {\it connected} embedded symplectic surface.
We begin with the elementary constraint from the $\omega-$symplectic genus. 

\begin{lemma}
 If $g_{\omega}(A)$ is
negative, then $A$ cannot be represented by a connected embedded symplectic surface. Thus, if $A$ has negative square and $k$ is a sufficiently large
integer, then $kA$ cannot be represented by a connected embedded symplectic surface. 
\end{lemma}
\begin{proof} The first claim is obvious since  the $\omega-$symplectic genus of a class $A$
is the genus of
any connected embedded symplectic surface representing $A$ and a connected surface
must have  non-negative genus. The second claim follows
from the domination of $K_{\omega}(kA)$ by $kA\cdot kA$  if $A\cdot A<0$ and
$k$ is large. 
\end{proof}

Many classes with positive squares may also have negative $g_{\omega}(A)$.
The following is an explicit example.

\begin{example} Let $M={\Bbb CP}^2\#2\overline{\Bbb CP}^2$
with a symplectic form $\omega$ in the class PD($\lambda H-E_1-E_2$) for some $\lambda>2$, where
$H$ is  the positive generator of $H_2({\Bbb CP}^2;{\Bbb Z})$ and
$E_1$ and  $E_2$ are the positive generators of the $H_2$ of the $\overline
{\Bbb CP}^2$. 
 Consider the class $A=3H-2E_1-2E_2$. It is $\omega-$positive and   has square 1. 
  As the symplectic canonical class $K_{\omega}$ is the Poincar\'e dual
to $-3H+E_1+E_2$, we have  $K_{\omega}(A)=-9+2+2=-5$
 and hence $g_{\omega}(A)=-1$. Thus $A$ has no connected symplectic representative. 
\end{example}

The $\omega-$symplectic genus can also be used to show that some classes with disconnected symplectic representatives do not have
connected representatives. The following is such an example. 

\begin{example} 
Let $M=S^2\times T^2$ with a product symplectic form $\omega$. The class  $A=2[S^2]$ is represented by two disjoint parallel
embedded symplectic spheres. These two symplectic spheres can be tubed to a smoothly  embedded sphere $C$ representing $A$. However, $C$ cannot be a $\omega-$symplectic sphere. 
As a matter of fact, there is no connected symplectic representative of $A$
  as $K_{\omega}([S^2])=-2$ and hence the $\omega-$symplectic genus of $A$
is equal to $-1$. 
\end{example}

The second obstruction is the following $H_1-$rank invariant $b(A)$
introduced by B. H. Li and the author.

\begin{definition}
Consider a map from a surface $\Sigma$ (possibly disconnected) 
to a smooth 4-manifold
$M$ representing the homology class $A$. The pull back of 
$H^1(M;{\Bbb R})$ is a subspace of $H^1(\Sigma;{\Bbb R})$. The rank of the skew-symmetric 
cup product restricted to this subspace is an even integer. Let $b(A)$ 
be half of that integer. 
\end{definition}

It is not difficult to show that $b(A)$ only depends on the class $A$ using
a bordism argument. In fact, W. Browder pointed out to us that 
 $b(A)$ is half  
the rank of the skew-symmetric pairing on $H^1(M;{\Bbb R})$ given by
$<a,b>_A=(a\cup b)(A)$. This is simply because, if $f:\Sigma\longrightarrow M$ is a map representing
the class $A$, then, for $a, b\in H^1(M;{\Bbb R})$, we have
$(f^*a\cup f^*b)([\Sigma])=(a\cup b)(A)$. 

With  this interpretation of
$b(A)$ we see that $b(A)\leq b_1(M)/2$. 
Notice also that if there  is a map from a connected surface $\Sigma$
to $M$ representing $A$, then $b(A)\leq g(\Sigma)$. From the latter inequality we obtain the following constraint.

\begin{lemma} 
If  $g_{\omega}(A)<b(A)$, then $A$ cannot be represented by a connected symplectic surface.
\end{lemma}
Here is an example where Lemma 3.5 can be applied. 

\begin{example}
Let $M=\Sigma\times \Sigma'$ be the product of two surfaces
of genus 2  with a product symplectic form $\omega'$. Choose a pair of
disjoint non-homologous circles $\gamma_1$ and $\gamma_2$ in $\Sigma$
and similar pair of circles $\gamma_1'$ and $\gamma_2'$ in $\Sigma'$. 
Then $T_1=\gamma_1\times \gamma_1'$ and $T_2=\gamma_2\times \gamma_2'$
are embedded Lagrangian tori in $M$. By an observation of Gompf in ~\cite{Go},
  $\omega'$ can be deformed to a symplectic form
$\omega$ making $T_1$ and $T_2$ $\omega-$symplectic. 
Let $T$ be the union of $T_1$ and $T_2$ and 
$A$ be the class $[T]$. Then $A\cdot A=0$ and
$$K_{\omega}(A)=K_{\omega'}(A)=(2[\Sigma]+2[\Sigma'])\cdot ([T_1]+[T_2])=0,$$
 hence
$g_{\omega}(A)=1$. On the other hand, $b(A)=2$, as $H_1(T;{\Bbb R})$ has rank 4 and injects into
$H_1(M;{\Bbb R})$. 
\end{example}

\begin{remark} We say a class $A$ is big if $A\cdot A>0$. 
We note that, for a big and $\omega-$positive class $A$ and  a large
integer $k$, the $\omega-$symplectic genus of $kA$ is dominated
by the  positive term $k^2A\cdot A$, while $b(kA)$ is always bounded
by $b_1(M)$. 
Thus the constraints from $g_{\omega}(A)$ and $b(A)$  disappear for a sufficiently large multiple
of a big and $\omega-$positive class. 
\end{remark}

From the remark above, to get a general existence result we should 
focus our attention to big and $\omega-$positive classes. It turns out there is yet another  
obstruction  for big and $\omega-$positive classes, 
which comes from the stable classes. 

 \begin{definition} A homology class $B$ is said to be 
stable if  $B$ is $J-$effective for any
$J\in {\mathcal J}_{\omega}$, i.e. it can be represented by a
$J-$holomorphic curve.
\end{definition}

In fact, it follows from the Gromov-Uhlenbeck compactness that
$B$ is stable if  $B$ is $J-$effective for a dense subset of
$J\in{\mathcal J}_{\omega}$. Moreover, it follows from the regularity theorem for 
pseudo-holomorphic curves (Proposition 7.1  in ~\cite{Ta1}), 
that a stable class of a minimal manifold must satisfy
$-K_{\omega}(B)+B\cdot B\geq 0$.

\begin{lemma} Suppose $A$ is  realized by an embedded symplectic surface,
each component with non-negative self-intersection. Then
 $\alpha=$PD($A$) is non-negative on any stable class.
\end{lemma}
 
\begin{proof} This is based on  the intersection property of
pseudo-holomporphic curves. We can easily construct an $\omega-$compatible almost complex structure $J$ 
such that 
 any component $C_i$ of $C$ is a connected embedded  $J-$holomorphic curve. By definition,  a stable class $S\in H_2(M;{\Bbb Z})$ is represented by
 a union of irreducible $J-$holomorphic curves $D=\cup_jD_j$, where  each $D_j$ might be singular and has multiplicity $m_j>0$. In any case, 
$C\cdot D=\sum_{i,j}m_jC_i\cdot D_j$. If $C_i\ne D_j$, then $C_i\cdot D_j\geq 0$ by the positivity of intersections of
distinct irreducible pseudo-holomorphic curves in ~\cite{Mc1}. If $C_i=D_j$, then $C_i\cdot D_j=C_i\cdot C_i$, which is also non-negative by assumption. 
Thus we have shown that $\alpha(S)=A\cdot S=C\cdot D$ is non-negative. 
\end{proof}

Currently, the only way to tell whether a class is stable or not  is to evaluate the  Taubes-Witten invariants or the Gromov-Witten invariants
of this class (see ~\cite{IP}). The Taubes-Witten invariants of a class
$A$, like the more well-known
Gromov-Witten invariants, also count pseudo-holomorphic curves of fixed genus
representing $A$, however the curves are now allowed to be disconnected.  
A class is called a TW class if some  Taubes-Witten invariant of this class is non-trivial. GW classes are defined in the same way. Certainly a TW class or a GW class is a stable class.

An important property of both the TW invariants and the GW invariants
 is that they are invariant
under deformation of the symplectic forms. We use this property
to give a different argument for the obstructions coming from the TW classes
(the same argument works for the GW classes as well).

  Suppose $W$ is a TW class and $A$ is represented by an embedded $\omega-$symplectic surface
$C$ with each component having  trivial or positive normal bundle. By the inflation process
in ~\cite{LM}, for any $t\geq 0$, the class $[\omega]+t$PD($A$) is represented
by a symplectic form $\omega_t$ deformation equivalent to $\omega$. Since TW
classes only depend on the deformation equivalence class of symplectic forms, 
$W$ is a still a TW class for any $\omega_t$ and therefore we have 
$\omega_t(W)>0$. 
However if  PD($A)(W)<0$, then $[\omega]+t$PD($A$) is negative on $W$ for
$t$ large. Thus we must have PD($A) (W)\geq 0$ to begin with.

An immediate consequence of Lemma 3.9 is the following
\begin{corollary} Suppose $A$ is an $\omega-$positive class with $A\cdot A\geq 0$
and is represented by a {\it connected} embedded symplectic surface, then
PD($A$) is positive on any stable class.
\end{corollary}

  As observed in ~\cite{Mc3}, the class  of a symplectic $-1$ sphere
is a Gromov-Witten class. 
 Therefore, if a $\omega-$positive homology class $A$ with $A\cdot A\geq 0$ has negative 
intersection with such a class, it  cannot be represented by a connected embedded symplectic surface.
Actually, if we notice that in Example 3.2, the class
$H-E_1-E_2$ is the class of symplctic $-1$ sphere
and $A\cdot (H-E_1-E_2)=-1$, we may reach the same conclusion by Corollary 3.10.

So far it is not clear to the author whether there 
are other obstructions: i.e. whether there is an $\omega-$positive class $A$
which satisfies $g_{\omega}(A)\geq b(A)$ and 
is positive on the stable classes but not representable by
a connected symplectic surface. 
 Vidussi ~\cite{V} showed that 
for  fibered 3-manifolds $N$, there are obstructions for the symplectic  4-manifold $M=S^1\times N$ coming from the Seiberg-Witten monopole classes of $N$.
But from Taubes' picture ~\cite{Ta1} linking solutions to the Seiberg-Witten
equations and the pseudo-holomorphic curves, it is possible these monopole classes of $N$ may give rise to stable classes
of $M$.


\subsection{Constructions}

Having discussed some obstructions to the existence of connected embedded symplectic
surfaces
we now describe several constructions. In view of Remark 3.7, we 
will focus on the large multiples of big and $\omega-$positive
classes whose Poincar\'e duals  are non-negative on stable classes.

The prototype of a big and  $\omega-$positive class which is  
positive on any stable  class is the Poincar\'e dual to the class of an integral symplectic form. For such classes we have the following beautiful result of Donaldson
already mentioned in \S2.

\begin{theorem} If PD($A$) is sufficiently close to the ray
generated by $[\omega]$, then a sufficiently large multiple of $A$ is represented
by a connected embedded symplectic surface. In particular, if $[\omega]$ is an integral class then sufficiently large multiples of PD($[\omega]$) are thus
represented. 
 \end{theorem}

 The starting observation for this theorem is that, in a symplectic vector space
$(V,\omega)$, a small perturbation of a symplectic subspace remains symplectic.  In particular,
a subspace is symplectic if it is close to
a complex subspace (for some $J\in {\mathcal J}_{\omega}(V)$). The folowing lemma
in ~\cite{Do} makes it precise. 

\begin{lemma} Let $f:{\Bbb C}^n\longrightarrow {\Bbb C}$ be 
${\Bbb R}-$linear with $f=a'+a''$, where $a'$ (respectively $a''$) is
${\Bbb C}-$linear (respectively ${\Bbb C}-$antilinear). If $|a''|<|a'|$,
then
$f$ is surjective and ker$(f)$ is a symplectic subspace of ${\Bbb
C}^n$. 
\end{lemma}
 
To describe Donaldson's construction we first assume that  $[\omega]$ is an integral class. In this case
there is a complex line bundle $L$ over $M$ with the $c_1(L)=[\omega]$, and the codimension
2 submanifolds are obtained as the zero sets $Q_s$ of suitable transverse sections $s$
of $L^{\otimes k}$ for large $k$. More precisely, fix a compatible almost complex structure $J$ on $M$ and choose a connection on the
line bundle $L$ whose curvature is
$\omega$. The connection on $L$ gives rise to  operators $\bar \partial_J$ and
$\partial_J$ on sections of $L$ and hence $L^{\otimes k}$. Given a section $s$ 
of $L^{\otimes k}$ and a point $x\in M$, $|\bar \partial_J s|$ and 
$|\partial_J s|$ 
when restricted to $x$ are respectively the holomorphic and
anti-holomorphic parts of the differential of $s$ at $x$, viewed as a map
between the complex vector spaces $T_xM$ and $T_{x,0}L^{\otimes k}$. 
Thus by Lemma 3.12 if there are sections such that 
$|\bar \partial_J s|<|\partial_J s|$ on $Q_s$ then 
$Q_s$ is symplectic. Such sections are called  approximately
holomorphic and are shown to exist by Donaldson. 
This construction clearly applies to symplectic manifolds with rational symplectic forms. As for a non-rational  symplectic form, we can always approximate it by 
rational symplectic forms. 

  The following well-known result of K\"ahler surfaces
suggests that we actually should be able to  go  far beyond the ray 
 generated by PD($[\omega]$). 

\begin{lemma} Suppose $\omega$ is a K\"ahler form
on a projective surface $M$, and 
  $A$ is  a big class Poincar\'e dual to 
 an integral cohomology class of type $(1,1)$. If we further assume
that PD($A$) is positive on all holomorphic curves, then a sufficiently
large multiple of $A$ is represented by a connected embedded symplectic surface.
\end{lemma}

Notice that here PD($A$) is   positive 
on all stable classes since it is assumed that it is positive on all holomorphic curves.  The proof of the lemma goes as follows: PD($A$) is an ample class by the Nakai-Moishezon criterion, 
thus, by Kodaira's embedding theorem,   sufficiently large multiples of $A$ are
Poincar\'e dual to very ample classes and the corresponding holomorphic line bundles have plenty of holomorphic sections whose zero loci are 
irreducible smooth holomorphic curves. 

In view of the two preceeding results, we would like to raise the following
question. 

\begin{question} Let $A$ be a big and $\omega-$positive class whose Poincar\'e dual is non-negative on any stable class. Then, is a sufficiently high multiple
of $A$  represented by a connected embedded symplectic 
surface?
\end{question}

To shed more light on this question let us presently make a general remark 
on the constructions of smooth 2-dimensional submanifolds. 
In general there are two ways to do so.  The first method, which we call `mapping into', starts with  mapping a 
surface into  $M$, and then deforming the map 
and possibly smoothing the image to an embedding. The second method, which we call
`mapping out', 
instead starts by choosing another manifold $N$ and a codimension 2 submanifold $S\subset N$, and
then taking the 
inverse image  of a map from  $M$ into $N$ which is transversal to $S$. 
For the `mapping out' constructions, $N$ is
often taken to be a complex line bundle over $M$.

The constructions of 2-dimensional symplectic submanifolds
 also follow the same routes. In fact, the constructions in Theorem 3.11 and Lemma 3.13 are clearly of the second kind. 

Let us presently turn to the `mapping into' constructions. 
In \S2 we start with a continuous map and then use transversality and 
the h-principles to deform it to a symplectic embedding. However, this method
fails in dimension 4 and we can only obtain symplectic immersions with both positive and negative double points.

As we already mentioned, only positively immersed symplectic surfaces can be
smoothed to an embedded symplectic surface. 
One way to obtain positively immersed symplectic surface is to
start with a connected embedded symplectic surface with positive self-intersection. Since being a closed symplectic submanifold is an open condition, we can
perturbe it to several nearby symplectic surfaces intersecting each other
transversally and positively. By smoothing the double points, we obtain the
following result.

\begin{lemma} If a big class $A$ is represented by a connected embedded symplectic surface, then for any positive integer $k$, the class $kA$ is 
also represented as such.
\end{lemma}

In fact, the same conclusion holds for a class with square zero
by the circle sum construction (see
~\cite{LiL1} for details).

In dimension 4,  the most effective `mapping into' approach to obtain a positively
immersed symplectic surface
is to construct a simple pseudo-holomorphic curve. Here a pseudo-holomorphic curve $u:\Sigma\longrightarrow
M$  is said to be simple
if the restriction of $u$ to any  of its components is not
multiply covered and no two components have the same image. 
The observation in ~\cite{Mc1} that any simple pseudo-holomorphic curve can be perturbed to  an immersed
pseudo-holomorphic curve $C$ of the same genus for some 
 nearby almost complex structure
readily leads to the folowing result.

\begin{lemma} If $A$ is represented by a simple
$J-$holomorphic curve $u:(\Sigma, j)\longrightarrow 
(M, J)$ for some $\omega-$tamed almost complex structure
$J$, then $A$ is represented by an embedded symplectic surface. 
\end{lemma}

In view of Lemma 3.16, we say a class $A$ is  simple if it is represented by
a simple $J-$holomorphic curve for some
$J\in {\mathcal J}_{\omega}$. 

There are abundant simple classes in a K\"ahler surface. Let us
recall some relevant facts here. 
In a K\"ahler surface $(M,\omega, J)$, holomorphic curves arise
 as the zero loci of  sections of holomorphic line bundles. If a holomorphic
line bundle is globally generated off a finite set of
points,
Bertini's Theorem will show  that the generic section is smooth
away from the base locus of the system.     Since the zero locus is pure codimension one (locally defined by one equation), this shows
then a generic divisor is a reduced curve. 
By the desingularization of curves, if $C$ is a reduced curve in 
an algebraic surface, then there exists a compact Riemann surface $\tilde C$ and a holomorphic map $\psi:\tilde C\longrightarrow C$ that is one-to-one over smooth points of $C$. 
Thus we have the following criterion of a simple class.
 
\begin{corollary} If ${\mathcal L}$ is a holomorphic line bundle
globally generated off a finite set of points, then PD($c_1({\mathcal L}))$ is
a simple class and hence can be 
represented by an embedded symplectic surface. 
\end{corollary}

In particular, if a holomorphic line bundle is generated by global sections,
then it is represented by connected embedded symplectic surfaces.
Kawamata's base-point-free theorem (\cite{Ka}) provides many such line bundles which are sufficiently high
powers of certain nef line bundles. Another general source is the following theorem of Reider (\cite{Re}):
 Let $M$ be a projective K\"ahler surface and ${\mathcal L}$ be
an ample line bundle on $M$. If  $c_1({\mathcal L})^2\geq 5$ and 
$c_1({\mathcal L})(\Gamma) \geq 2$ for 
all irreducible curves $\Gamma\subset M$, then ${\mathcal K}\otimes {\mathcal L}$ is globally generated. In particular, ${\mathcal K}\otimes {\mathcal L}^3$ is globally generated for
any ample line bundle ${\mathcal L}$.

Finally, we describe the construction originated from the Seiberg-Witten 
theory. On a symplectic 4-manifold, given a cohomology class $e$, there is an
associated  Spin$^c$ structure ${\mathcal L}_e$ whose 
$c_1$ is equal to $-K_{\omega}+2e$. Witten observed that, on a K\"ahler surface, 
if  the equations are deformed by positive  multiples of K\"ahler forms, then
the solutions  
correspond exactly to holomorphic sections of a holomophic line bundle with $c_1=e$,
therefore giving rise to holomorphic curves representing $e$.
 
Taubes vastly generalizes this picture to symplectic 4-manifolds. 
He ~\cite{Ta1} starts by fixing a compatible almost complex structure on a symplectic 4-manifold
$(M,\omega)$ and then uses the induced metric to define
the  Seiberg-Witten equations. He is able to prove that, if
the Seiberg-Witten invariant is non-trivial, then the solutions to the  Seiberg-Witten equations for the Spin$^c$ structure ${\mathcal L}_e$ deformed by 
a large multiple of the symplectic form gives rise to sections of the complex line
bundle
with $c_1$ equal to $e$. Moreover the zero loci are (possibly disconnected) pseudo-holomorphic subvarieties.
He further shows that, by imposing 
$d(e)=K_{\omega}\cdot e+e\cdot e$ number of  generic point constraints,
 for a generic choice of a compatible almost complex structure, 
the pseudo-holomorphic subvarieties satisfying the constraints are
essentially embedded pseudo-holomorphic submanifolds.     
Motivated by this remarkable result, Taubes ~\cite{Ta2} defines a Gromov type invariant
counting 
embedded pseudo-holomorphic curves in a fixed class (in the connected case, an earlier
attempt was made  in ~\cite{R}), 
which we call the 
Gromov-Taubes invariant.  
The final piece in this grand picture is that the Gromov-Taubes
invariant of PD($e$) is the same as 
the Seiberg-Witten invariant for the Spin$^c$ structure ${\mathcal L}_e$
at least when $b^+>1$.

On a symplectic 4-manifold with $b^+>1$, one consequence of the Taubes-Seiberg-Witten
theory is 
that the symplectic canonical class $K_{\omega}$ is 
a Gromov-Taubes class, hence is always realized by
an embedded symplectic surface. Moreover, if $e$ is realized, so is
$K_{\omega}-e$. 
  
When $b^+$ is 1, together  with the wall crossing formula in ~\cite{LLiu3}, we are able to prove in ~\cite{LLiu2}
that most big and $\omega-$positive classes in a minimal symplectic 4-manifold
with $b^+=1$ are represented by a connected embedded
symplectic surface. More precisely, we have 

\begin{proposition} Let $(M,\omega)$ be a minimal symplectic $4-$manifold
with
$b^+=1$. Let $A$ be a big and $\omega-$positive class. If $A-$PD($K_\omega$) is also
$\omega-$positive and has non-negative square, then
$A$ is represented by a connected symplectic surface.
In particular,  for
$N$ big,  $NA$ is represented by a connected symplectic
surface.
\end{proposition}

When $M$ is not minimal, the same is true if we
take into account the obstructions coming from ${\mathcal E}_{\omega}$, which is  the set of
the exceptional classes represented by symplectic $-1$ spheres.

\begin{proposition}
 Let $(M,\omega)$ be a symplectic $4-$manifold with
$b^+=1$ and symplectic canonical class $K_{\omega}$.
Let $A$ be a big and $\omega-$positive class. 
Assume that $A-$PD($K_\omega$) is $\omega-$positive and has non-negative square.  Further assume that
$A\cdot E\geq -1$ for all
$E\in {\mathcal E}_{\omega}$.
Then  $A$ can be represented by an embedded symplectic surface.
Furthermore, if $A\cdot E\geq 0$ for all $E\in {\mathcal E}_{\omega}$, then the symplectic surface is
connected.
\end{proposition}

See also ~\cite{DS} and ~\cite{S} for a  purely symplectic approach
to some of the consequences of the Taubes-Seiberg-Witten theory
building  on the existence of Lefschetz pencils by ~\cite{Do2}.

\begin{remark}
With a combination of the constructions above it is  
shown in ~\cite{LiL1} that, in fact every big and $\omega-$positive class of a symplectic 
$S^2-$bundle is represented by a connected embedded symplectic surface.
\end{remark}

\medskip
In summary the situations where Question 3.14 has an affirmative answer
are the following:

\noindent 1. $b^+=1$.

\noindent 2. $A$ is already represented by a connected embedded symplectic surface.

\noindent 3. PD($A$) is close to the ray generated by $[\omega]$. 

\noindent 4. PD($A$) is  an ample class of a projective K\"ahler surface.


\medskip

\begin{thebibliography}{99}







\bibitem  {A} D. Auroux, \textit{Asymptotically holomorphic families of symplectic submanifolds},
Geom. Funct. Anal. 7 (1997), 971-995.

\bibitem {B} P. Biran, \textit{Symplectic packings in dimension 4},
Geom. Funct. Anal. 7 (1997), no.3. 420-437.


\bibitem {Do} S. Donaldson, \textit{Symplectic submanifolds and almost-complex geometry},
J. Differential Geom. 44 (1996), 666-705.

\bibitem {Do2} S. Donaldson, \textit{Lefschetz pencils on symplectic manifolds}, J. Differential Geom. 53 (1999), 205-236.

\bibitem {DS} S. Donaldson, I. Smith, \textit{Lefschetz pencils and the canonical class for symplectic
4-manifolds}, Topology 42 (2003), 743-785.

\bibitem {EM} Y. Eliashberg, N. Mishachev, \textit{Introduction to
the h-principle}, GSM 48, American Mathematical Society, 2002.

\bibitem {FS} R. Fintushel, R. Stern, \textit{Symplectic surfaces in a fixed homology class},  J. Differential Geom.  52  (1999),  no. 2, 203--222. 

\bibitem {Go} R. Gompf, \textit{A new construction of symplectic manifolds}, Ann. of Math. 142 (1995), 527-595.




\bibitem {Hi} M. Hirsch, \textit{Immersions of manifolds}, Trans. Amer. Math. Soc.  93  (1959) 242--276. 

\bibitem {Ho} H. Hopf, \textit{Fundmentalgruppe and zweite Bettische Gruppe},
Comment. Math. Helv. 14 (1942), 257-309.

\bibitem {IP} E. Ionel, T. Parker, \textit{Gromov-Witten invariants of symplectic sums},
Math. Res. Letters 5 (1998), 563-576. 

\bibitem {Ka} Y. Kawamata, \textit{A generalization of Kodaira-Ramanujam's vanishing theorem},
 Math. Ann.  261  (1982), no. 1, 43--46.

\bibitem {K} U. Koschorke, \textit{Complex and real vector bundle monomorphism}, 
Topology and its Applications 91 (1999), 259-271.

\bibitem {KM} P. Kronheimer and T. Mrowka, \textit{The genus of embedded
surfaces in the projective plane}, Math. Res. Letters 1 (1994), 797-808.

\bibitem {Le} H. V. Le, \textit{Realizing homology classes by symplectic submanifolds}.
MPI preprint Nr. 61/2004.

\bibitem {LiL1} B. H. Li, T. J. Li, \textit{The symplectic circle sum}, in preparation. 


\bibitem {LLiu1} T. J. Li, A. K. Liu, \textit{Symplectic
structures on ruled surfaces and a generalized adjunction inequality}, Math.
Res. Letters 2 (1995), 453-471.


\bibitem {LLiu2} T. J. Li, A. K. Liu,  \textit{Uniqueness of symplectic 
canonical class, surface
cone and symplectic cone of $4-$manifolds with $b^+=1$},
J.  Differential Geom.  58  (2001), 331-370. 

\bibitem {LLiu3} T. J. Li,  A. K. Liu, \textit{Family Seiberg-Witten invariants
and wall crossing formulas}, Comm. in Analysis and Geometry 9 (2001), no. 4, 777-823

\bibitem {LM} F. Lalonde, D.  McDuff, \textit{The classification of ruled 
symplectic 4-manifolds}, Math. Res. Letters 3 (1996), 769-778.

\bibitem {Mc1} D. McDuff, \textit{The structure of rational and ruled
symplectic $4$-manifolds}. J. Amer. Math. Soc. 3 (1990), no. 3, 679--712.


\bibitem {Mc2} D. McDuff, \textit{The local behavior of holomorphic curves in almost complex  4-manifolds}, J. Differential Geom. 34 (1991), 679-710.

\bibitem {Mc3} D. McDuff, \textit{Immersed spheres in symplectic $4-$manifolds}, 
Ann. Inst. Fourier (Grenoble) 42 (1992), no.1-2, 369-392.



\bibitem  {MPS} V. Munoz, F. Presas, I. Sols, \textit{Almost holomorphic embeddings in Grassmannians
with applications to singular symplectic submanifolds}, J. Reine Angew. Math. 547 (2002), 149-189.

\bibitem {MST} J. Morgan, Z. Szabo, C. Taubes, \textit{A product formula and the generalized Thom conjecture}, 
J. Differential Geom. 44 (1996), 706-788. 

\bibitem {MW} J. McCarthy, J. Wolfson, \textit{Symplectic normal connect sum},  Topology  33  (1994),  no. 4, 729--764. 

\bibitem {OS} P. Oszvath, Z. Szabo, \textit{The symplectic Thom conjecture},
Ann. of Math. (2) 151 (2000), no. 1, 93-124.



\bibitem {P} R. Paoletti, \textit{Symplectic subvarieties of projective fibrations over
symplectic manifolds}, Ann. Inst. Fourier Grenoble 49 (1999), 1661-1672.

\bibitem {Re} I. Reider, \textit{Vector bundles of rank $2$ and linear systems on algebraic surfaces},
 Ann. of Math. (2)  127  (1988),  no. 2, 309--316.

\bibitem {R}  Y. Ruan, \textit{Symplectic topology and complex
surfaces}, Geometry
and analysis on complex manifolds, 171--197, World Sci.
Publishing, River Edge, NJ, 1994.


\bibitem {S} I. Smith, \textit{Serre-Taubes Duality 
for pseudoholomorphic curves}, Topology 42 (2003), 931-979.


\bibitem {ST} B. Siebert, G. Tian, On the holomorphicity of genus two Lefschetz fibration, preprint.

\bibitem {V} S. Vidussi, \textit{Norms on the cohomology of a 3-manifold and Seiberg-Witten
theory}, Pacific J. Math. 208 (2003), no.1, 169-186. 

\bibitem {Ta1} C. H. Taubes,  \textit{SW$\Rightarrow$Gr: From Seiberg-Witten
equations to pseudo-holomorphic curves}, J. Amer. Math. Soc. 9 (1996) 845-918.

\bibitem {Ta2} C. H. Taubes, \textit{Counting pseudo-holomorphic
submanifolds in dimension 4}, J. Differential Geom. 44 (1996) 818-893.


\bibitem {Th} R. Thom, \textit{Quelques propri\'et\'es globales des
vari\'et\'es
diff\'erentiables}, Comment. Math. Helv. 28 (1954), 17-86. 

\end{thebibliography}
\end{document}